\documentclass[a4paper]{article}

\usepackage[all]{xy}\usepackage[latin1]{inputenc}        
\usepackage[dvips]{graphics,graphicx}
\usepackage{amsfonts,amssymb,amsmath,color,mathrsfs, amstext}
\usepackage{amsbsy, amsopn, amscd, amsxtra, amsthm,authblk, enumerate}
\usepackage{upref}
\usepackage{geometry}
\usepackage[displaymath]{lineno}
\usepackage{float}
\usepackage{yhmath}
\usepackage[colorlinks,
            linkcolor=red,
            anchorcolor=red,
            citecolor=red
            ]{hyperref}

\def\epsilon{\varepsilon}

\DeclareMathOperator{\loc}{loc}
\DeclareMathOperator{\argmin}{argmin}

\newcommand{\G}{\mathrm{grad}}

\newtheorem{theorem}{Theorem}[section]
\newtheorem{lemma}{Lemma}[section]
\newtheorem{proposition}{Proposition}[section]
\newtheorem*{proposition*}{Proposition}
\newtheorem{corollary}{Corollary}[section]
\newtheorem*{corollary*}{Corollary}
\newtheorem{definition}{Definition}[section]
\newtheorem*{definitions*}{Definitions}
\newtheorem*{conjecture*}{\bf Conjecture}

\newtheorem*{example*}{\bf Example}
\theoremstyle{remark}
\newtheorem{remark}{\bf Remark}[section]

\newtheorem{assumption}{Assumption}[section]

\numberwithin{equation}{section}

\begin{document}

\title{A discretization of Caputo derivatives with application to time fractional SDEs and gradient flows}

\author[1]{Lei Li\thanks{leili2010@sjtu.edu.cn}}
\author[2]{Jian-Guo Liu\thanks{jliu@phy.duke.edu}}
\affil[1]{School of Mathematical Sciences, Institute of Natural Sciences, MOE-LSC, Shanghai Jiao Tong University, Shanghai, 200240, P. R. China.}
\affil[2]{Department of Mathematics and Department of Physics, Duke University, Durham, NC 27708, USA.}

\date{}

\maketitle

\begin{abstract}
We consider a discretization of Caputo derivatives resulted from deconvolving a scheme for the corresponding Volterra integral. Properties of this discretization, including signs of the coefficients, comparison principles, and stability of the corresponding implicit schemes, are proved by its linkage to Volterra integrals with completely monotone kernels. We then apply the backward scheme corresponding to this discretization to two time fractional dissipative problems, and these implicit schemes are helpful for the analysis of the corresponding problems. In particular, we show that the overdamped generalized Langevin equation with fractional noise has a unique limiting measure for strongly convex potentials and establish the convergence of numerical solutions to the strong solutions of time fractional gradient flows. The proposed scheme and schemes derived using the same philosophy can be useful for many other applications as well.
\end{abstract}

\section{Introduction}

Continuous time fractional calculus has been used widely in physics and engineering for memory effect, viscoelasticity, porous media etc \cite{kst06, diethelm10,lll17}. Among them, the Caputo's and Riemann-Liouville's definitions are very popular for power law memory kernels \cite{caputo1967linear, kst06, diethelm10,liliu2018}. Caputo's definition of fractional derivatives was first introduced in \cite{caputo1967linear} to study the memory effect of energy dissipation for some anelastic materials and soon became a useful modeling tool in engineering and physical sciences for nonlocal interactions in time (see \cite{kouxie04,fle2,fle4}). Compared with Riemann-Liouville derivatives, Caputo derivatives remove the singularities at the origin and are suitable for initial value problems \cite{liliu2018}.  There are other models that have power law kernels with certain cutoffs (especially exponential cutoffs) so that they can give transitions between different behaviors \cite{vergara2008,sandev2017generalized,liemert2017generalized,chen2018tempered,molina2018crossover}. In this paper, we are interested in discretizing gradient type time fractional dissipative problems. Moreover, we desire to use the numerical discretization to investigate the properties of the solutions of the time continuous dissipative problems.  The first problem is the time fractional stochastic differential equation (fractional SDE) of dissipative type, which is the overdamped limit of the generalized Langevin equation with fractional noise.  Another problem is the time fractional gradient flows in a separable Hilbert space. 

In some complex systems, the Langevin equations cannot give accurate predictions and the interaction between the system and the surrounding heat bath can no longer be modeled by white noise \cite{kubo66,kouxie04}. The generalized Langevin equation (GLE)
\begin{gather}\label{eq:gle}
\begin{split}
&\dot{X}=v,\\
& m\dot{v}=-\nabla V(X)-\int_{t_0}^t\gamma(t-s)v(s)\,ds+\eta(t)
\end{split}
\end{gather} 
was then proposed by Mori and Kubo \cite{mori65,kubo66} to describe such complex systems with memory. In this model, $-\int_{t_0}^t\gamma(t-s)v(s)\,ds$ is the friction acting on the system we consider and $t_0$ is the point where the memory is counted from (usually $0$ or $-\infty$). The friction is the mean effect of the interaction between the system and the heat bath. The last term $\eta(t)$ is the noise which is the fluctuation part of the interaction between the system and the heat bath. Later, the GLE was recovered by dimension reduction from Ford-Kac and Kac-Zwanzig models using Mori-Zwanzig projection (\cite{zwanzig73,coarsegrain1_eric,coarsegrain3_karniadakis,leimkuhler2018}).  
In the GLE models, the noise $\eta$ and the kernel for the friction $\gamma(\cdot)$ satisfy the so-called fluctuation-dissipation theorem (FDT)
\begin{gather}\label{eq:fdt}
\mathbb{E}(\eta(t) \eta(t+\tau))= k T \gamma(|\tau|),~ \forall \tau\in \mathbb{R}.
\end{gather}
Intuitively, the random force and the frictional kernel all originate from the interaction between the system and the surrounding environment. When the energy balance is reached, they must be related for the system to achieve the correct temperature. In \cite{kouxie04}, Kou and Xie considered the GLE with fractional Gaussian noise to explain the subdiffusive behaviors for a protein molecule in solution. Later, this model was studied by many authors \cite{kou08,deng2009}. The fractional Gaussian noise is the distributional derivative of the fractional Brownian motion $B_H$ (see \cite{nualart06}  and section \ref{sec:formalfsde} for more details)
\begin{gather}
\eta=\sigma \dot{B}_H(t).
\end{gather}
Using FDT \eqref{eq:fdt} and considering the overdamped limit, we obtain the fractional SDE as the overdamped GLE (see section \ref{sec:formalfsde} for a simple derivation and the rigorous definition):
\begin{gather}
D_c^{2-2H}X=-\nabla V(X)+\sigma dB_H,
\end{gather}
where $D_c^{\alpha}$ is the Caputo derivative (see section \ref{sec:setup} for more explanation). 
In \cite{lll17}, the fractional SDE has been studied theoretically. If the force is linear, it was shown that the process converges in law to a unique limiting measure. Moreover, if the FDT is satisfied, the limiting measure is the Gibbs measure. The general potential $V$ cases seem hard to justify. In  \cite{fang2018}, numerical methods have been designed for the overdamped GLE and the numerical results there give positive evidence. One of our goals in this paper is to use the numerical schemes to prove that the limiting measure is unique if the potential $V$ is strongly convex.

Though there might not be strong physical interpretation, the time fractional gradient flow is of its own mathematical interest and can be used for new phase field models (see \cite{chen2002phase,liu2003phase} for the phase field models). In particular, consider a separable Hilbert space $H$ and a functional $\phi: H\to \mathbb{R}$ that is lower semicontinuous. The time fractional gradient flow we consider is
\begin{gather}\label{eq:strongsol}
D_c^{\alpha}u\in -\partial\phi(u),~~u(0)=u_0,
\end{gather}
where the Frechet subdifferential $\partial\phi$ at $u$ is a set defined as
\begin{gather}
\partial \phi(u):= \left\{ \xi\in H: 
\liminf_{w\to u}\frac{\phi(w)-\phi(u)-\langle \xi, w-u\rangle }{|w-u|} \ge 0 \right\}.
\end{gather}
If $\partial \phi$ contains a single point $\xi$, then we define $\G\phi(u):=\xi$. We aim to investigate the discretization using our scheme in this paper and establish error estimates.  For related fractional gradient flow, one can see \cite{vergara2008} where the memory kernel takes the form $t^{-\gamma}e^{-\mu t}$ with exponential decay. If $\phi$ is convex, then $\partial\phi$ is accretive and some related Volterra equations have been discussed in \cite{cn78,cn81}, where the existence of generalized solutions have been established using the Yosida approximations. The equation we will consider is not included in these papers.

Numerical discretizations of time fractional differential equations and related equations have already been investigated by many authors \cite{lubich1986,linxu2007,lixu2009space,nochetto2016,jin2017correction,jin2018discrete,atangana2018new,shens2019a,shens2019b}.  In particular, the authors of \cite{lixu2009space,nochetto2016} applied the $L^1$ schemes, which approximate the Caputo derivative directly, for several dissipative problems. In \cite{jin2017correction}, some corrections are made for the first $k-1$ steps so that the non-smoothness at $t=0$ does not pollute the desired accuracy of the schemes. In \cite{shens2019a,shens2019b}, some spectral methods have been developed for fractional differential equations. Moreover, in \cite{fllx18}, some comparison principles for the discrete fractional equations have been established.
Unfortunately, using these discretizations to study the fractional SDE and time fractional gradient flows is not appropriate because the time continuous problems are not well understood yet. 
Our approach is to consider the discretization of the integral formulation first and apply the deconvolution (see \cite{liliu2018note}) to obtain the discretization of the Caputo derivatives in differential form. Since the integral formulation is more suitable for passing the limit, we are then able to conclude the important results regarding the time-continuous problems and establish some error estimates. Note that the new scheme is not just discretization of Volterra integrals since some important properties will be proved based on the de-convolved sequence, which seems very hard using the discretization of the integral form. Besides the problems considered in this paper, the scheme proposed here or schemes derived
using the same philosophy may be applied for other problems \cite{molina2018crossover,owolabi2017mathematical}.

The rest of the paper is organized as follows. In section \ref{sec:setup}, we give the basic notations and propose the discretization of the Caputo derivatives using deconvolution. In section \ref{sec:property}, we prove some important properties of the new discretization. Section \ref{sec:fsde} and section \ref{sec:fgradientflow} are devoted to fractional SDE and time fractional gradient flows.  In particular, we show that the overdamped GLE with fractional noise has a unique limiting measure for strongly convex potentials; we also establish some error estimates for the strong solutions of time fractional gradient flows.

\section{Notation and setup for the discretization}\label{sec:setup}

Let $B$ be a Banach space. Consider the following equation for a mapping: $X: [0, T]\to B$:
\begin{gather}\label{eq:caputoexp}
D_c^{\alpha}X(t)=f(t),
\end{gather}
where $f: [0, T]\to B$ is some mapping. $D_c^{\alpha}$ represents the Caputo derivative
of order $\alpha\in (0, 1)$ (\cite{kst06, diethelm10}). If $X(\cdot)$ is regular enough, for example, absolutely continuous, the Caputo derivative traditionally is defined as
\begin{gather}\label{eq:caputotradition}
D_c^{\alpha}X(t)=\frac{1}{\Gamma(1-\alpha)}\int_0^t \frac{\dot{X}(s)}{(t-s)^{\alpha}}\,ds.
\end{gather}
In \cite{liliu2018,liliu2018compact}, a generalized definition of Caputo derivative based on convolution groups was proposed.  To explain this generalized defintion, we first recall the distributions $\{g_\beta\}$ in \cite{liliu2018}:
\begin{gather}\label{eq:g}
g_{\beta}(t)=
\begin{cases}
\frac{1}{\Gamma(1+\beta)}D\left(\theta(t)t^{\beta}\right),& \beta\in (-1, 0)\\
\frac{1}{\Gamma(\beta)}t_+^{\beta-1}, & \beta>0.
\end{cases}
\end{gather} 
Here $\theta(t)$ is the standard Heaviside step function, $\Gamma(\cdot)$ is the gamma function, $t_+=\theta(t)t=\max(t, 0)$, and $D$ means the distributional derivative on $\mathbb{R}$. Indeed, $g_{\beta}$ can be defined for $\beta\in \mathbb{R}$ (see \cite{liliu2018}) so that $\{g_{\beta}: \beta\in\mathbb{R}\}$ forms a convolution group. In particular, we have 
\begin{gather}\label{eq:grouprelation}
g_{\beta_1}*g_{\beta_2}=g_{\beta_1+\beta_2}.
\end{gather}
 Note that the support of $g_{\beta_i}$ ($i=1,2$) is bounded from left, so the convolution is well-defined. 
\begin{definition}[\cite{liliu2018,liliu2018compact}] \label{def:caputo}
Let $0<\alpha<1$. Consider $X\in L_{\loc}^1([0, T),B)$. Given $X_0\in B$, we define the $\alpha$th order generalized Caputo derivative of $X$, associated with initial value $X_0$, to be a distribution as $D_c^{\alpha}X: C_c^{\infty}(-\infty, T; \mathbb{R})\to B$ with support in $[0, T)$, given by 
\begin{gather}\label{eq:generalizeddef}
D_c^{\alpha}X=g_{-\alpha}*\Big((X-X_0)\theta(t)\Big).
\end{gather}
If$\ \lim_{t\to 0+}\frac{1}{t}\int_0^t\|X(s)-X_0\|_Bds=0$, we call $D_c^{\alpha}X$ the Caputo derivative of $X$.
\end{definition}
The weak Caputo derivatives  in \cite{liliu2018compact} for mappings in general Banach spaces was defined through a dual equality using right derivatives. One can verify easily that 
the one in \cite{liliu2018compact}  agrees with Definition \ref{def:caputo}.
This generalized definition appears complicated. However, it is theoretically more convenient, since it allows us to take advantage of the underlying group structure. In fact, making use of the convolutional group structure \eqref{eq:grouprelation} (see \cite{liliu2018} for more details), it is straightforward to convert \eqref{eq:caputoexp} with \eqref{eq:generalizeddef} into the Volterra type equation
\begin{gather}\label{eq:Voltter}
X(t)=X_0+\frac{1}{\Gamma(\alpha)}\int_0^t (t-s)^{\alpha-1}f(s)\,ds.
\end{gather}
Indeed \eqref{eq:Voltter} is well-known for regular enough $f$; see \cite[Lemma 2.3]{df02}. The theory in  \cite{liliu2018,liliu2018compact} tells us that it still holds for $f$ to be distributions. 

For absolutely continuous functions, Definition \ref{def:caputo} reduces to \eqref{eq:caputotradition}. 
In this paper, we sometimes need the generalized definition, Definition \ref{def:caputo}, and its equivalence to \eqref{eq:Voltter} since we need to consider the Caputo derivative of a continuous function later.

For numerical setup, we fix the terminal time $T$ and consider time step
\begin{gather}
k=T/N.
\end{gather}
Define $t_n=nk$. We will use $X_n$ to represent the numerical solution at $t_n$.

\subsection{Discretization of the fractional derivatives: two options}

Depending on whether we discretize \eqref{eq:caputoexp} or \eqref{eq:Voltter}, we can possibly have different schemes (see \cite[Section 6]{fllx18} for some relevant discussions).
Discretization of \eqref{eq:Voltter} and deconvolution yields a discretization of the Caputo derivative, whose implicit scheme turns out very useful for studying two important time fractional dissipative problems. In particular, we can conclude the asymptotic behavior of the fractional SDEs and study the time fractional gradient flows in separable Hilbert spaces ( see sections \ref{sec:fsde} and \ref{sec:fgradientflow} respectively).

\subsubsection{Discretization of the differential form}

Discretizing \eqref{eq:caputoexp} directly is well studied in literature (see \cite{diethelm1997,linxu2007}). The $L^1$ scheme in \cite{diethelm1997,linxu2007} is widely used in applications due to the good sign of the coefficients (see \cite{nochetto2016,liumazhou2017}). The scheme is given by
\begin{gather}
(\bar{\mathcal{D}}^{\alpha}X)_n
=k^{-\alpha}(\bar{c}_0X_n-\bar{c}_1X_{n-1}-\ldots-\bar{c}_n^n X_0).
\end{gather}
Here, the coefficients are given by
\begin{gather}
\begin{split}
& \Gamma(2-\alpha)\bar{c}_0=1,\\
& \Gamma(2-\alpha)\bar{c}_j=-((j+1)^{1-\alpha}-2j^{1-\alpha}+(j-1)^{1-\alpha}),~~1\le j\le n-1, \\
& \Gamma(2-\alpha)\bar{c}_n^n=(n^{1-\alpha}-(n-1)^{1-\alpha})
\end{split}
\end{gather}
We have the following observations:
(i) $\bar{c}_j>0$,  $\bar{c}_n^n>0$; (ii) $\bar{c}_0-\sum_{j=1}^{n-1}\bar{c}_j-\bar{c}_n^n=0$;
(iii) 
\begin{gather}
\bar{c}_j=\frac{-1}{\Gamma(-\alpha)}j^{-1-\alpha}\left(1+O\left(\frac{1}{j}\right)\right),~
j\to\infty,\  \ \ 
\bar{c}_n^n=\frac{n^{-\alpha}}{\Gamma(1-\alpha)}\left(1+O\left(\frac{1}{n}\right)\right).
\end{gather}

\subsubsection{Discretization of the integral form and deconvolution}
Alternatively, we can consider the discretization of the integral form \eqref{eq:Voltter} and then take deconvolution to get the approximation for 
the Caputo derivative. In fact, discretizing the integral form has been well-studied in literature (see for example \cite{lubich1986,galeone2009}). Slightly different from the discretizations in these works, what we choose to do is to approximate $f$ with piecewise constant functions. Then, we take deconvolution and get the approximation to the differential form. 

To start, we approximate $f(t)$ by
\begin{gather}
\tilde{f}(t)=F_n,~~t\in (t_{n-1}, t_n].
\end{gather}
Then, \eqref{eq:Voltter} gives the following scheme
\begin{gather}\label{eq:discretefracint}
X_n-X_0=k^{\alpha}\sum_{m=1}^{n}a_{n-m}F_m =: (J_kF)_n ,
\end{gather}
where the right hand side is a discrete integral and the sequence $a$ is given by
\begin{gather}\label{eq:seqa}
a=(a_0, a_1, \ldots, a_n,\ldots)= \frac{1}{\Gamma(1+\alpha)}(1, 2^{\alpha}-1, 3^{\alpha}-2^{\alpha}, \ldots).
\end{gather}

For convenience, we define $F_0=0$ and introduce the sequence $F\in B^{\mathbb{N}}$ by
\begin{gather}
F=(F_0, F_1, \ldots, F_n, \ldots),
\end{gather}
so that for $n\ge 0$, $X_n-X_0=k^{\alpha}(a*F)_n$. The convolution between $u$ and $v$ is given by
\begin{gather}
(u*v)_n=\sum_{m=0}^n u_m v_{n-m}=(v*u)_n.
\end{gather}

Let $a^{(-1)}$ be the convolution inverse of $a$ such that
\begin{gather}
a*a^{(-1)}=a^{(-1)}*a=\delta_d :=(1, 0, 0,\ldots).
\end{gather}
Then, we obtain for $n\ge 0$ that
$k^{-\alpha}(a^{(-1)}*(X-X_0))_n=F_n$. We therefore obtain a new scheme for discretizing the Caputo derivative
\begin{gather}\label{eq:capnewdis}
(\mathcal{D}^{\alpha}X)_n=k^{-\alpha}(a^{(-1)}*(X-X_0))_n.
\end{gather}
Though equivalent to the discretization of Volterra integral, we regard this as a new scheme because some important properties (e.g. Theorem \ref{thm:implicitcomp} (2)-(3) and \eqref{eq:fractionalgradientkeyestimate}) will be proved based on this differential form \eqref{eq:capnewdis}, which will be hard using the integral form \eqref{eq:discretefracint}.

\section{Properties of the discretization}\label{sec:property}

In this section, we discuss in detail the properties of discretization \eqref{eq:capnewdis}.
One can refer to \cite{liliu2018note} for some discussion of using deconvolution to define discrete fractional calculus.

We first introduce some definitions for the discussion. We say a sequence $v=(v_0, v_1, \ldots)$ is completely monotone if $((I-S)^jv)_k\ge 0$ for any $j\ge 0, k\ge 0$ where $(Sv)_j=v_{j+1}$.  A sequence is completely monotone if and only if it is the moment sequence of a Hausdorff measure (a finite nonnegative measure on $[0,1]$)  (\cite{widder41}). Another description is given below in Lemma \ref{lmm:generating}.  The generating function of a sequence $v=(v_0, v_1, \ldots)$ is defined by 
\begin{gather}
F_v(z)=\sum_{n=0}^{\infty} v_n z^n. 
\end{gather}
Another concept we introduce is the Pick function. A function $f:\mathbb{C}_+\to\mathbb{C}$ (where $\mathbb{C}_+$ denotes the upper half plane, not including the real line) is Pick if it is analytic such that $\mathrm{Im}(z)>0\Rightarrow \mathrm{Im}(f(z))\ge 0$.  Now, we state some properties of sequences in terms of the generating functions, for which we omit the proofs.
\begin{lemma}\label{lmm:generating}
\begin{enumerate}[(1)]
\item For convolution,  $F_{u*v}(z)=F_u(z)F_v(z)$, and
$F_{v^{(-1)}}(z)=(F_v(z))^{-1}$.

\item (\cite[Corollary VI.1]{fs09}) Assume $F_v(z)$ is analytic on $\Delta :=\{z: |z|<R, z\neq 1, |\mathrm{arg}(z-1)|>\theta\}$, for some $R>1, \theta\in (0, \frac{\pi}{2})$. If $F_v(z)\sim (1-z)^{-\beta}$ as $z\to 1, z\in \Delta$ for $\beta\neq 0, -1, -2, -3, \ldots$, then $v_n \sim \frac{1}{\Gamma(\beta)}n^{\beta-1},~n\to \infty$.

\item $\lim_{n\to\infty}v_n=\lim_{z\to 1^-}(1-z)F_v(z)$.

\item (\cite{lp16}) A sequence $v$ is completely monotone if and only if the generating function $F_v(z)=\sum_{j=0}^{\infty}v_jz^j$ is a Pick function that is analytic and nonnegative on $(-\infty, 1)$. 
\end{enumerate}
\end{lemma}

For convenience, define a sequence $c=(c_0, c_1, \ldots, c_n,\ldots)$ as (see \eqref{eq:seqa} for $a$)
\begin{gather}
c_0=a_0^{(-1)},~~c_i=-a_i^{(-1)},~\forall i\ge 1.
\end{gather}
Moreover, it is convenient to introduce 
\begin{gather}
c_n^n=c_0-\sum_{i=1}^{n-1}c_i.
\end{gather}
Then,  \eqref{eq:capnewdis} can be reformulated as
\begin{gather}\label{eq:capnewdis2}
\begin{split}
(\mathcal{D}^{\alpha}X)_n & =k^{-\alpha}\Big(c_0(X_{n}-X_0)-\sum_{i=1}^{n-1}c_i (X_{n-i}-X_0)\Big)\\
&=k^{-\alpha}\Big(c_0X_{n}-\sum_{i=1}^{n-1}c_i X_{n-i}
-c_n^n X_0\Big).
\end{split}
\end{gather}
Using the result in \cite{liliu2018note}, we have the following claims.
\begin{proposition}\label{pro:coefficients}
Consider scheme \eqref{eq:capnewdis2}. The following claims hold:
\begin{enumerate}[(1)]
\item $c_i>0$ for $i\ge 0$ and $c_0=\sum_{i=1}^{\infty}c_i=\Gamma(1+\alpha)$.
Consequently, $c_n^n=\sum_{i=n}^{\infty}c_i>0$ and $c_0=\sum_{i=1}^{n-1}c_i+c_n^n$.

\item We have the following asymptotics for the coefficients.
\begin{gather}
c_j=\frac{-1}{\Gamma(-\alpha)}j^{-1-\alpha}\left(1+o(1) \right),~
j\to\infty,\  \ \ 
c_n^n=\frac{n^{-\alpha}}{\Gamma(1-\alpha)}\left(1+o(1) \right).
\end{gather}

\end{enumerate}
\end{proposition}

\begin{proof}
(1). First of all, we recall that (see \eqref{eq:g} for the definition of $g_{\alpha}$) 
\[
a_n=\int_{n}^{n+1}g_{\alpha}(t)\,dt.
\]
Since $g_{\alpha}(\cdot)$ is completely monotone (which means $(-1)^m \frac{d^m}{dt^m}g_{\alpha}(t)\ge 0$ for any $t\in (0,\infty)$), then the sequence $a=\{a_n\}$ is completely monotone.
Since $a^{(-1)}=(c_0, -c_1, -c_2, \ldots)$,  by \cite[Theorem 2.3]{liliu2018note}, one has that
$c_0>0$ while $(c_1, c_2, \ldots)$ is a completely monotone sequence. 
The sign of $c_i$ is thus proved.

Using the explicit formula for $a_n$,  $F_{a}(z)\to \infty$ as $z\to 1^-$. Hence, we find that 
$F_{a^{(-1)}}(z)\to 0,~~z\to 1^-$.
Noting the sign of elements for $a^{(-1)}$, the monotone convergence theorem holds and thus  $\sum_{i=0}^{\infty}a_i^{(-1)}=0$. The other claims then follow accordingly.

(2). We consider the function
\[
H(z)=F_a(z)-(1-z)^{-\alpha}=: \sum_{n=0}^{\infty} d_n z^n.
\]
By \cite[Theorem VI.1]{fs09}, we have
$[z^n](1-z)^{-\alpha}=\frac{1}{\Gamma(\alpha)}n^{\alpha-1}(1+O(\frac{1}{n}))$, $n\to\infty$,
where $[z^n]F(z)$ means the coefficient of $z^n$ in the series expansion of $F(z)$ about $0$.
Hence, $|d_n|\le C\frac{1}{n^{2-\alpha}}$, and $H(z)$ is a locally bounded function (bounded on any compact set). Hence,
\[
F_{a^{(-1)}}(z)=\frac{(1-z)^{\alpha}}{1+(1-z)^{\alpha} H(z)}.
\]
Clearly, this function is analytic in $\{z: |z|<1+\epsilon, z\neq 1, |\mathrm{arg}(z-1)|>\frac{\pi}{4}\}$ for some $\epsilon>0$. 
Applying the second claim in Lemma \ref{lmm:generating} gives the asymptotics for $c_j=-a_j^{(-1)}$ ($j\ge 1$).
Using the fact $c_n^n=\sum_{i=n}^{\infty}c_i$, the asymptotics for $c_n^n$ then follows.

\end{proof}

The discrete comparison principles are important for stability of numerical schemes. Below, we prove several important comparison criteria that are helpful for the stability of the implicit schemes. For the stability of some explicit schemes, one may refer to \cite{fllx18}. 

\begin{theorem}\label{thm:implicitcomp}
Consider discretization \eqref{eq:capnewdis2}. Let $u=\{u_n\}$, $v=\{v_n\}$ and $w=\{w_n\}$ be three sequences in $\mathbb{R}^{\mathbb{N}}$, with $u_0\le v_0\le w_0$.
\begin{enumerate}[(1)]

\item (Convex functional) Suppose $E(\cdot): \mathbb{R}^d\to \mathbb{R}, X\mapsto E(X)$ is convex. Then,
\[
(\mathcal{D}^{\alpha}E(X))_n \le (\mathcal{D}^{\alpha}X)_n\cdot\nabla E(X_n).
\]

\item (Comparison principle for nonincreasing $f$) Suppose $f(s, \cdot)$ is non-increasing. Assume $u,v,w$ satisfy the discrete implicit relations
\[
(\mathcal{D}^{\alpha}u)_n\le f(t_n, u_n),~~(\mathcal{D}^{\alpha}v)_n = f(t_n, v_n),~~
(\mathcal{D}^{\alpha}w)_n \ge f(t_n, w_n).
\]
Then,  $u_n\le v_n \le w_n$.

\item (Comparison principle for Lipschitz $f$) Assume $f$ is Lipschitz continuous in the second variable with Lipschitz constant $L$. If
\[
(\mathcal{D}^{\alpha}u)_n\le f(t_n, u_n),~~(\mathcal{D}^{\alpha}v)_n = f(t_n, v_n),~~
(\mathcal{D}^{\alpha}w)_n \ge f(t_n, w_n),
\]
then for step size $k$ with $c_0>k^{\alpha}L$, $u_n\le v_n \le w_n$.

\item (Comparison principle for integral form)
Assume $f(s,\cdot)$ is non-decreasing and Lipschitz continuous in the second variable with Lipschitz constant $L$. Introduce $f_u, f_v, f_w$ by, for example, $f_u=(0, f(t_1,u_1), f(t_2,u_2), \ldots)$. If
\begin{gather*}
u_n \le u_0+(J_k f_u)_n,~~
 v_n =v_0+(J_k f_v)_n,~~ w_n \ge w_0+(J_k f_w)_n,
\end{gather*}
 then for step size $k$ with $k^{\alpha}a_0L=k^{\alpha}L/c_0<1$, $u_n\le v_n \le w_n$.
\end{enumerate}
\end{theorem}
\begin{proof}
(1).  By \eqref{eq:capnewdis2}, Proposition \ref{pro:coefficients} and the convexity of $E(\cdot)$, we have
\begin{gather*}
\begin{split}
(\mathcal{D}^{\alpha}X)_n\cdot\nabla E(X_n) & =\nabla E(X_n)\cdot 
k^{-\alpha}\left( \sum_{i=1}^{n-1}(X_n-X_{n-1})+c_n^n(X_n-X_0)\right)\\
& \ge k^{-\alpha}\left(\sum_{i=1}^{n-1}(E(X_n)-E(X_{n-1}))+c_n^n(E(X_n)-E(X_0))\right)
=(\mathcal{D}^{\alpha} E(X))_n.
\end{split}
\end{gather*}

(2). Let $\xi_n=u_n-v_n$. Then,
\[
\mathcal{D}^{\alpha}\xi \le f(t_n, u_n)-f(t_n, v_n).
\]
Multiplying $1(\xi_n\ge 0)$ on both sides, and defining $\eta_n=\xi_n \vee 0=\max(\xi_n, 0)$, we have
\[
1(\xi_n \ge 0)k^{-\alpha} \Big(c_0\xi_n-\sum_{i=1}^{n-1}c_i \xi_{n-i}
-c_n^n\xi_0 \Big) \le 1(\xi_n \ge 0)(f(t_n, u_n)-f(t_n, v_n))\le 0.
\]
Since $\xi_n 1(\xi_n\ge 0)=\xi_n\vee 0=\eta_n$, $\xi_i 1(\xi_n\ge 0)\le \xi_i\vee 0=\eta_i$, one easily finds that
\[
(\mathcal{D}^{\alpha}\eta)_n \le 1(\xi_n \ge 0)(\mathcal{D}^{\alpha}\xi)_n \le 0.
\]
Since $\eta_0= 0$, one easily finds $\eta_n\le 0$, and hence $u_n\le v_n$.  It is similar to compare $v_n$ and $w_n$.

(3). We compare $u_n$ with $v_n$. We know already $u_0\le v_0$. Now, suppose $n\ge 1$ and assume for all $m\le n-1$, we have proved $u_m\le v_m$ already. We now consider $m=n$. 
\[
k^{-\alpha}c_0(u_n-v_n)\le (\mathcal{D}^{\alpha}(u-v))_n
\le f(t_n, u_n)-f(t_n, v_n)\le L |u_n-v_n|.
\]
If $u_n> v_n$, we then have $(k^{-\alpha}c_0-L)(u_n-v_n)\le 0$ which is clearly not true. Hence, induction shows that the claim is true for all $n$.
Comparing the sequence $v$ with $w$ is similar and we omit.

(4).  Direct computation shows that
\[
u_n-v_n\le u_0-v_0+k^{\alpha}a_0(f(t_n, u_n)-f(t_n, v_n))+k^{\alpha}\sum_{m=1}^{n-1}a_{n-m}(f(t_m, u_m)-f(t_m, v_m)).
\]
If we have proved that $u_m\le v_m$ for $m\le n-1$, then 
$u_n-v_n\le k^{\alpha}a_0 L|u_n-v_n|$.
The proof then follows by induction similarly as in 3. Comparing the sequence $v$ with $w$ is similar and we omit.
\end{proof}

We now consider the stability of the implicit scheme applied to the simple FODEs
\begin{gather}\label{eq:simpleode}
D_c^{\alpha}X=\lambda X,~~X(0)=x_0>0,
\end{gather}
whose solution is given by $X(t)=x_0 E_{\alpha}(\lambda t^{\alpha})$, where 
\[
E_{\alpha}(z)=\sum_{n=0}^{\infty}\frac{z^n}{\Gamma(n\alpha+1)}
\]
is the Mittag-Leffler function \cite{HMS}.
\begin{theorem}\label{thm:schemeforlinear}
Consider the implicit scheme applied on the fractional ODE \eqref{eq:simpleode}:
\begin{gather}\label{eq:simplefodescheme}
(\mathcal{D}^{\alpha}X)_n=\lambda X_n \Leftrightarrow X_n=X_0+\lambda k^{\alpha}\sum_{m=1}^{n}a_{n-m}X_m.
\end{gather}
\begin{enumerate}[(1)]
\item If $\lambda>0$ and $k^{\alpha}\lambda<c_0$, then $X(t_n) \le X_n \le X_{n+1}$. 
If otherwise $\lambda<0$, 
$\lim_{n\to\infty}X_n=0$.

\item Consider $\lambda>0$. Suppose $k_i,i=1,2$ satisfy $k_i^{\alpha}\lambda<c_0$ and $k_1=2^{m_1} k_2$ for some $m_1\in\mathbb{N}$. Let $X_n^{(i)}$ be the numerical solutions. Define the piecewise constant functions $\bar{X}_i(t)$ by $\bar{X}_i(t)=X_n^{(i)},~t\in (t_{n-1}^{(i)}, t_n^{(i)}]$ for $i=1,2$. Then,
$\bar{X}_1(t)\ge \bar{X}_2(t)$.
Consequently, there exists a constant $C(\alpha, T)>0$ such that  for any $k$ with $k^{\alpha}\lambda\le \frac{1}{2}c_0$, 
\begin{gather}\label{eq:stabilityineqa}
\sup_{n: nk\le T}X_n\le C(\alpha, T)X_0.
\end{gather}

\item When $k$ is sufficiently small, $\sup_{n: nk\le T}|X_n - X(t_n)|\le C_1(\alpha, T) k^{\alpha}$.
\end{enumerate}
\end{theorem}

\begin{proof}

(1). Consider $\lambda>0$. The induction formula from the differential form  reads
\[
(c_0-k^{\alpha}\lambda)X_n=\sum_{i=1}^{n-1}c_i X_{n-i}+c_n^n X_0.
\]
If $n=1$, $c_1^1=c_0$. Then, we clearly have $X_1=\frac{c_0}{c_0-k^{\alpha}\lambda}X_0>X_0$.
Suppose we have proved $X_m\ge X_{m-1}$ for $m\le n$ with $n\ge 1$. For $n+1$,
\[
(c_0-k^{\alpha}\lambda)X_{n+1}= \sum_{i=1}^{n-1}c_i X_{n+1-i}+(c_n X_1+c_{n+1}^{n+1}X_0)
\ge \sum_{i=1}^{n-1}c_i X_{n-i}+(c_n X_0+c_{n+1}^{n+1}X_0).
\]
Since $c_n+c_{n+1}^{n+1}=c_n^n$, the claim then follows.

Now, consider the equivalent integral form (second in \eqref{eq:simplefodescheme}).
\[
X_n=X_0+\frac{\lambda}{\Gamma(\alpha)}\sum_{j=1}^{n}\int_{t_{j-1}}^{t_j}(t_n-s)^{\alpha-1}X_{j}\,ds.
\]
The accurate solution satisfies
\[
X(t_n)=X_0+\frac{\lambda}{\Gamma(\alpha)}\sum_{j=1}^{n}\int_{t_{j-1}}^{t_j}(t_n-s)^{\alpha-1}X(s)\,ds
\le X_0+\frac{\lambda}{\Gamma(\alpha)}\sum_{j=1}^{n}\int_{t_{j-1}}^{t_j}(t_n-s)^{\alpha-1}X(t_j)\,ds.
\]
By the third claim in Theorem \ref{thm:implicitcomp},  $X(t_n)\le X_n$.

Now, we consider $\lambda<0$. Recall that $F_0=0$ and
\[
X_n-X_0=k^{\alpha}\lambda a*(X-X_0\delta_{n0}).
\]
The generating function of $X^n$ is thus given by
\[
F_X(z)=X_0\frac{(1-z)^{-1}-k^{\alpha}\lambda F_a(z)}{1-k^{\alpha}\lambda F_a(z)}.
\]
As $z\to 1^-$, $F_a(z)\to\infty$, $(1-z)F_a(z)\to 0$, and hence
\[
\lim_{n\to\infty}X_n=\lim_{z\to 1^-}(1-z)F_X(z)=0.
\]

(2). We only need to consider $m_1=1$ (or $k_1=2k_2$).  By (1), the piecewise constant functions $\bar{X}_i(t)$'s are nondecreasing. Suppose that for $n\ge 1$,  one has
$\bar{X}_1(t)\ge \bar{X}_2(t)$, $t\in [0, (n-1)k_1]$. Then, for $t\in ((n-1)k_1, nk_1]$, one only needs $\bar{X}_1(nk_1)\ge \bar{X}_2(nk_1)=\bar{X}_2(2nk_2)$ since $\bar{X}_2$ is nondecreasing. By the integral formulation (second in \eqref{eq:simplefodescheme}), 
\begin{align*}
\bar{X}_1(nk_1)& =X_0+\frac{\lambda}{\Gamma(\alpha)}\int_0^{t_{n-1}^{(1)}}(nk_1-s)^{\alpha-1}\bar{X}_1(s)\,ds 
+\frac{\lambda}{\Gamma(\alpha)}\int_{t_{n-1}^{(1)}}^{t_n^{(1)}}(nk_1-s)^{\alpha-1}\bar{X}_1(nk_1)\,ds\\
&\ge X_0+\frac{\lambda}{\Gamma(\alpha)}\int_0^{(n-1)k_1}(nk_1-s)^{\alpha-1}\bar{X}_2(s)\,ds+\frac{\lambda k_1^{\alpha}}{\Gamma(1+\alpha)}\bar{X}_1(nk_1)
\end{align*}

On the other hand,
\begin{align*}
\bar{X}_2(nk_1)=X_0+\frac{\lambda}{\Gamma(\alpha)}\int_0^{t_{n-1}^{(1)}}(nk_1-s)^{\alpha-1}\bar{X}_2(s)\,ds+\frac{\lambda}{\Gamma(\alpha)}\int_{t_{n-1}^{(1)}}^{t_n^{(1)}}(nk_1-s)^{\alpha-1}\bar{X}_2(s)\,ds.
\end{align*}
The last term is simply controlled by $\frac{\lambda k_1^{\alpha}}{\Gamma(1+\alpha)}\bar{X}_2(2nk_2)$ due to monotonicity of $\bar{X}_2$.
Since $c_0=\Gamma(1+\alpha)$ and $\bar{X}_2(2nk_2)=\bar{X}_2(nk_1)$, we then find
\[
\bar{X}_1(nk_1)\ge \frac{1}{1-k_1^{\alpha}\lambda/c_0}\left(X^0+\frac{\lambda}{\Gamma(\alpha)}\int_0^{(n-1)k_1}(nk_1-s)^{\alpha-1}\bar{X}_2(s)\,ds \right)\ge \bar{X}_2(2nk_2).
\]

We now prove the stability. For any step size $k$, we choose $k_0=2^m k$ such that $k_0^{\alpha}\lambda 
\in (\frac{c_0}{2^{1+\alpha}}, \frac{c_0}{2}]$.
Then, with time step $k_0$, there are 
\[
N_0=\frac{T}{k_0} \le T (2^{1+\alpha}\lambda/c_0 )^{1/\alpha}
\]
steps. Then consider the induction using the differential form:
\[
(c_0-k_0^{\alpha}\lambda)X_n \le \sum_{i=1}^{n-1}c_i X_{n-m}+c_n^n X_0
\le c_0 X_{n-1}.
\]
Hence, $X_n\le 2X_{n-1}\le 2^{N_0}X_0$.
The claim then follows.

(3). By the explicit formula of the solution for \eqref{eq:simpleode}, we know $X$ is $\alpha$-H\"older continuous and for $t>0$, it is smooth. Inserting $X(\cdot)$ into the integral form, we have
\begin{gather*}
X(t_n)=X_0+\frac{\lambda}{\Gamma(\alpha)}\sum_{m=1}^n \int_{t_{m-1}}^{t_m}(t_n-s)^{\alpha-1}X(s)\,ds
=X_0+\lambda k^{\alpha}\sum_{m=1}^n a_{n-m}X(t_m)+R_n,
\end{gather*}
where
\[
R_n=\frac{\lambda}{\Gamma(\alpha)}\sum_{m=1}^n \int_{t_{m-1}}^{t_m}(t_n-s)^{\alpha-1}(X(s)-X(t_m))\,ds,
\]
and thus
\[
|R_n|\le C(T)k^{\alpha}\frac{\lambda}{\Gamma(\alpha)}\sum_{m=1}^n \int_{t_{m-1}}^{t_m}(t_n-s)^{\alpha-1}\,ds
=C_1(T,\alpha)k^{\alpha}.
\]
Hence, the error $E^n:=|X_n-X(t_n)|$ satisfies
\[
E^n\le k^{\alpha}|\lambda|\sum_{m=1}^n a_{n-m}E_m+Ck^{\alpha}.
\] 
Using the comparison principle for integral formulation in Theorem \ref{thm:implicitcomp} and the stability result \eqref{eq:stabilityineqa}, we have for $k$ sufficiently small that $E^n\le C(\alpha, T) k^{\alpha}$.
\end{proof}

By Theorem \ref{thm:schemeforlinear}, the following claims hold when we compare the numerical solution with the exact solutions of some FODEs.
\begin{corollary}
Suppose $f(\cdot)\in C^2[A,\infty)$ is nondecreasing and globally Lipschitz for some $A\in\mathbb{R}$. Let $u(\cdot)$ be the solution to the FODE $D_c^{\alpha}u=f(u)$ with $u(0)=U>A$ and $f(U)>0$.  Let $\{u_n\}$ be the numerical solution of the implicit scheme $(\mathcal{D}^{\alpha}u)_n=f(u_n),~~u_0=u(0)$.
Then, for $k$ sufficiently small, $u(t_n)\le u_n \le u_{n+1}$. Moreover, for any $T$ such that $u$ exists on $[0, T]$, we have for some $C(T)>0$ that
\begin{gather}
\sup_{n: nk\le T}|u(t_n)-u_n|\le C(T)k^{\alpha}.
\end{gather}
\end{corollary}
\begin{proof}
The solution to the FODE satisfies the following (\cite{fllx18}):
\begin{gather}\label{eq:utnineq}
u(t_n) =u_0+\frac{1}{\Gamma(\alpha)}\sum_{j=1}^{n}\int_{t_{j-1}}^{t_j}(t_n-s)^{\alpha-1}f(u(t_j))\,ds+R_n,
\end{gather}
where
\[
R_n=\frac{1}{\Gamma(\alpha)}\sum_{j=1}^{n}\int_{t_{j-1}}^{t_j}(t_n-s)^{\alpha-1}\Big(f(u(s))-f(u(t_j))\Big)\,ds.
\]
By the theory in \cite{fllx18}, $u(\cdot)$ is nondecreasing and thus $R_n\le 0$. Consequently, applying Theorem \ref{thm:implicitcomp} (4), we have $u(t_n)\le u_n$. 

To prove that $\{u_n\}$ is nondecreasing, we use induction. It is clear that $u_0\le u_1$. Now, assume we have proved $u_0\le u_1\le \cdots \le u_{n}$ for $n\ge 1$. We now prove $u_n\le u_{n+1}$. Using the equivalent integral form, we find
\begin{gather}\label{eq:unineq}
\begin{split}
u_{n}-k^{\alpha}a_0f(u_n) &=u_0+\frac{1}{\Gamma(\alpha)}\sum_{j=1}^{n-1}\int_{t_{j-1}}^{t_j}(t_{n+1}-s)^{\alpha-1}f(u_j)\,ds \\
&\le u_0+\frac{1}{\Gamma(\alpha)}\sum_{j=2}^{n}\int_{t_{j-1}}^{t_j}(t_{n+1}-s)^{\alpha-1}f(u_j)\,ds \\
& \le u_{n+1}-k^{\alpha}a_0 f(u_{n+1}).
\end{split}
\end{gather}
This implies that $u_n\le u_{n+1}$ when $k$ is sufficiently small.

Finally, by \eqref{eq:utnineq} and \eqref{eq:unineq}, we have
\[
|u(t_n)-u_n|\le \frac{1}{\Gamma(\alpha)}\sum_{j=1}^{n}\int_{t_{j-1}}^{t_j}(t_n-s)^{\alpha-1}|f(u(t_j))-f(u_j)|\,ds
+|R^n|
\]
It is well-known that $u$ is $\alpha$-H\"older continuous on $[0, T]$ (one can for example combine \cite[Lemma 3.1]{fllx18} and \cite[Theorem 3.1]{skm93}). Consequently, $\sup_{n: nk\le T}|R_n|\le C(T)k^{\alpha}$. Hence,
\[
|u(t_n)-u_n|\le \frac{L}{\Gamma(\alpha)}\sum_{j=1}^{n}|u(t_j)-u_j|\int_{t_{j-1}}^{t_j}(t_n-s)^{\alpha-1}\,ds
+C(T)k^{\alpha}.
\]
Applying Theorem \ref{thm:schemeforlinear} (2), we thus find $\sup_{n: nk\le T}|u(t_n)-u_n|\le C(T)k^{\alpha}$.
\end{proof}

\section{Limiting behavior of fractional SDE}\label{sec:fsde}

In this section, we use the implicit scheme corresponding to the discretization \eqref{eq:capnewdis2} to study the fractional SDE as advertised in the introduction. 
In particular, we first of all provide some details for the derivation of the FSDE, and then prove that when the potential is strongly convex, there is a unique limiting measure for the FSDE.

\subsection{A formal derivation of the fractional SDE}\label{sec:formalfsde}

We first of all derive the auto-correlation function for the fractional noise.  Recall that the fractional Brownian motion has the following: 
\begin{gather}\label{fb:cov}
\mathbb{E}(B_t^HB_s^H)=R_H(s, t) := \frac{1}{2}\left(s^{2H}+t^{2H}-|t-s|^{2H}\right).
\end{gather}
Fix $\tau\neq 0$. Formally, for $t>0$ with $t+\tau>0$, it holds that 
\begin{multline}
\mathbb{E}(\dot{B}_H(t)\dot{B}_H(\tau+t))=\lim_{h\to 0, h_1\to 0}\mathbb{E}\left(\frac{B_H(t+h_1)-B_H(t)}{h_1}\frac{B_H(t+\tau+h)-B_h(t+\tau)}{h}\right)\\
  =\lim_{h\to 0, h_1\to 0}
  \frac{1}{2hh_1}\bigl(|\tau+h|^{2H}-|\tau+h-h_1|^{2H}-|\tau|^{2H}+|\tau-h_1|^{2H}\bigr)=H(2H-1) |\tau|^{2H-2}.
\end{multline}
Assume there is no extra singularity for $\tau=0$, we check formally for $s<t$:
\begin{gather}\label{fb:cov2}
\begin{split}
\mathbb{E}(B_t^HB_s^H) &=\int_0^t\int_0^s \mathbb{E}\left(\dot{B}_H(z)\dot{B}_H(w) \right) \,dz dw \\
&=\int_0^s\int_0^tH(2H-1)|z-w|^{2H-2}dzdw =\frac{1}{2}(s^{2H}+t^{2H}-(t-s)^{2H}).
\end{split}
\end{gather}
Now that \eqref{fb:cov2} agrees with \eqref{fb:cov}. Hence, the assumption for no extra singularity at $\tau=0$ is reasonable.
According to FDT \eqref{eq:fdt}, the GLE \eqref{eq:gle}
 is then reduced to the following dimensionless equation
\[
\epsilon\dot{v}=-\nabla V-\frac{1}{\Gamma(2H-1)}\int_0^t(t-s)^{2H-2}v(s)\,ds
+\frac{\sqrt{2}}{\sqrt{\Gamma(2H+1)}}\dot{B}_H.
\]
Here, $\epsilon=\frac{mT^{2H}}{\gamma_0}$ with $T$ being the scale for time, $\gamma_0$ being typical scale for the friction (see \cite{lll17} for detials). In the overdamped regime, $\epsilon\ll 1$, the  GLE with fractional Gaussian noise formally corresponds to the fractional SDE
\begin{gather}
D_c^{\alpha}X=-\nabla V(X)+\sigma \dot{B}_H,
\end{gather}
with $\alpha=2-2H, ~\sigma=\frac{\sqrt{2}}{\sqrt{\Gamma(2H+1)}}$.
This overdamped generalized Langevin equation (overdamped GLE) is rigorously defined through the following integral formulation
\begin{gather}\label{eq:fsde2}
X(t)=X_0+\frac{1}{\Gamma(\alpha)}\int_0^t (t-s)^{\alpha-1}b(X(s))\,ds+G(t)
\end{gather}
where
\begin{gather}
G(t) :=\frac{\sigma}{\Gamma(\alpha)}\int_0^t (t-s)^{\alpha-1} dB_H(s).
\end{gather}
We can consider generally $\alpha\in (0, 1)$ and $\sigma>0$ for fractional SDEs. Of course, only the one with $\alpha=2-2H$ has physical significance, which is the overdamped GLE. It has  been shown in \cite{lll17} that when $\alpha=2-2H$ and $\sigma=\frac{\sqrt{2}}{\sqrt{\Gamma(2H+1)}}$, $G(t)$ is another fractional Brownian motion with Hurst parameter $1-H$ up to some multiplicative constant:  $G(t)\sim \beta_H B_{1-H}$ with $\beta_H=\frac{\sqrt{2}}{\sqrt{\Gamma(3-2H)}}$. 

When the force $-\nabla V(x)$ is linear, the distribution of $X$ converges algebraically to the Gibbs measure (\cite{lll17}). For general cases, whether it converges to the Gibbs measure is unknown.   Recently, in the case of overdamped GLE, some numerical experiments indicate that the law of $X$ still converges algebraically to the corresponding Gibbs measure for general potential (\cite{fang2018}).

\subsection{Convergence to equilibrium for strongly convex potentials}
In this subsection, we will try to use our discretization to study the limit behaviors of the FSDE for the strongly convex potential $V$. In particular, we show that there is a unique limiting measure as $t\to\infty$. Letting
\begin{gather}
b(x)=-\nabla V(x),
\end{gather}
we will assume the following.
\begin{assumption}
There exists some $\mu>0$ such that
\begin{gather}
(x-y)\cdot (b(x)-b(y))\le -\mu |x-y|^2,~\forall x, y\in\mathbb{R}^d.
\end{gather}
Moreover, $b(\cdot)$ is Lipschitz continuous so that for some $L>0$,
\begin{gather}\label{eq:lip}
|b(x)-b(y)|\le L|x-y|.
\end{gather}
\end{assumption}

The Lipschitz condition of $b$ may be relaxed by proving that the probability density of the process decays fast at infinity. Since this is not our focus, we assume the Lipschitz condition for simplicity.
 As proved in  \cite{lll17} and \cite{fang2018}, with assumption \eqref{eq:lip}, the fractional SDE \eqref{eq:fsde2} has a unique strong solution and for any $T>0$ there exists $C(T)>0$ such that the following hold.
\begin{gather}\label{eq:momentholder}
\sup_{t\ge T} \mathbb{E}|X(t)|^2\le C(T),\quad \sup_{t\ge T}\sqrt{\mathbb{E}|X(t+\delta)-X(t)|^2}\le C(T)\delta^{H+\alpha-1}.
\end{gather}

Now, we consider the limiting behavior of the law for $X(t)$. 
Given two different initial data $X^{(i)}(0)$, $i=1,2$, we consider the strong solutions of \eqref{eq:fsde2}. 
We will use the synchronization coupling to compare the distributions of the two processes.
Taking the difference between two solutions, we have
\[
X^{(1)}(t)-X^{(2)}(t)=X^{(1)}(0)-X^{(2)}(0)+\frac{1}{\Gamma(\alpha)}\int_0^t (t-s)^{\alpha-1}(b(X^{(1)}(s))-b(X^{(2)}(s)))\,ds.
\]
To get the idea of a proof, we apply the theory in \cite{liliu2018} so that the Caputo derivatives can be defined pathwise for $X^{(1)}-X^{(2)}$ (see also Section \ref{sec:setup} for the brief introduction). In the distributional sense, it holds that 
\begin{gather}
D_c^{\alpha}(X^{(1)}-X^{(2)})(t)=b(X^{(1)}(t))-b(X^{(2)}(t)),~\text{almost surely}.
\end{gather}
If $X^{(1)}-X^{(2)}$ is regular enough, applying \cite[Proposition 3.11]{liliu2018}, we have
\begin{gather}\label{eq:continuousineq}
\frac{1}{2}D_c^{\alpha}|X^{(1)}-X^{(2)}|^2(t)\le (X^{(1)}-X^{(2)})\cdot(b(X^{(1)})-b(X^{(2)}))\le -\mu |X^{(1)}-X^{(2)}|^2.
\end{gather}
If we define $u(t):=\mathbb{E}\Big(|X^{(1)}(t)-X^{(2)}(t)|^2\Big)$, it then holds that
\[
D_c^{\alpha}u(t)\le -2\mu u(t).
\]
Applying the comparison principle for $D_c^{\alpha}u=f(t, u)$ with nonincreasing $f(t,\cdot)$ (see, for example, \cite[Theorem 2.1]{fllx18note}) yields
\begin{gather}\label{eq:desiredcontrol}
u(t)\le u(0)E_{\alpha}(-2\mu t^{\alpha}).
\end{gather}
If this is true, we then are able to compare the laws of the two strong solutions of \eqref{eq:fsde2}  under Wasserstein-$2$ distance. Recall that the Wasserstein-$2$ distance is given by \cite{santambrogio2015}
\begin{gather}\label{eq:W2}
W_2(\mu, \nu)=\left(\inf_{\gamma \in \Pi(\mu,\nu)}\int_{\mathbb{R}^d\times\mathbb{R}^d}|x-y|^2 d\gamma\right)^{1/2},
\end{gather}
where $\Pi(\mu,\nu)$ is the set of joint distributions whose marginal distributions are $\mu$ and $\nu$ respectively. Equation \eqref{eq:desiredcontrol} will imply the convergence of the law of the process to the unique limiting measure.

The issue in the above argument is that \eqref{eq:continuousineq} is not justified rigorously. In the following, we shall utilize the implicit scheme based on discretization \eqref{eq:capnewdis2} to prove the convergence of the law.  In fact, we have the following claims.
\begin{theorem}\label{thm:limitsde}
Suppose assumption \eqref{eq:lip} holds and $X^{(i)}(t)$ are the two strong solution to the FSDE \eqref{eq:fsde2} with initial data $X^{i}(0)=X_0^{(i)}\sim \mu_0^{(i)}$ ($i=1,2$), where $\mu_0^{(i)}$ are some given probability measures.
Then, the laws of $X^{(i)}(t)$ satisfy in Wasserstein-$2$ distance that
\begin{gather}
W_2(\mu^{(1)}(t), \mu^{(2)}(t))\le W_2(\mu_0^{(1)}, \mu_0^{(2)})\sqrt{E_{\alpha}(-2\mu t^{\alpha})}.
\end{gather}
Consequently, the FSDE model has a unique limiting measure $\pi$.
\end{theorem}

We will apply the following backward Euler scheme based on \eqref{eq:capnewdis2} to FSDE \eqref{eq:fsde2}.
\begin{gather}\label{eq:fsdescheme}
X_n=X(0)+k^{\alpha}\sum_{m=1}^n a_{n-m} b(X_m)+G(t^n).
\end{gather}
We need some preparation for the complete proof. The first is the following convergence result of the scheme \eqref{eq:fsdescheme}.
\begin{lemma}\label{lmm:convresult}
Suppose assumption \eqref{eq:lip} holds and $X(t)$ is the unique strong solution to \eqref{eq:fsde2}.
Let $X_n$ be the numerical solution to \eqref{eq:fsdescheme}.
Then, for $k$ with $k^{\alpha}L<c_0/2$, 
\begin{gather}
\sup_{n: nk\le T}\sqrt{\mathbb{E}(|X_n-X(nk)|^2)} \le C(\alpha, T) k^{\alpha+H-1}.
\end{gather}
\end{lemma}
\begin{proof}
The proof is very similar to that for the third claim in Theorem \ref{thm:schemeforlinear}.
In fact, the strong solution of \eqref{eq:fsde2} satisfies
\begin{gather}\label{eq:localtruncation}
X(t_n)=X_0+k^{\alpha}\sum_{m=1}^{n}a_{n-m} b(X(t_m))+R_n,
\end{gather}
where
\[
R_n:=\frac{1}{\Gamma(\alpha)}\sum_{m=1}^n \int_{t_{m-1}}^{t_m}(b(X(s))-b(X(t_m))).
\]
Using \eqref{eq:lip} and \eqref{eq:momentholder}, one finds
\[
(\mathbb{E}|R_n|^2)^{1/2}\le C k^{\alpha+H-1}.
\]
Taking the difference between \eqref{eq:fsdescheme} and \eqref{eq:localtruncation} and defining 
$E_n:=(\mathbb{E}|X^n-X(t_n)|^2)^{1/2}$,  one then has
\begin{gather}
E_n \le k^{\alpha}L\sum_{m=1}^n a_{n-m} E_m +Ck^{\alpha+H-1}.
\end{gather}
Finally, using the comparison principle for integral formulation in Theorem \ref{thm:implicitcomp} and the stability result \eqref{eq:stabilityineqa}, the claim follows.
\end{proof}

Consider two numerical solutions $\{X^{(1)}_n\}$ and $\{X^{(2)}_n\}$ with initial data $X_0^{(i)}$ ($i=1,2$), with the synchronization coupling. The variable $Z_n:=X^{(1)}_n - X^{(2)}_n$ satisfies the following relation:
\begin{gather}
Z_n=Z_0+k^{\alpha}\sum_{m=1}^n a_{n-m} (b(X^{(1)}_m)-b(X^{(2)}_m)).
\end{gather}
Equivalently, one has almost surely that
\[
k^{-\alpha} \left(\sum_{j=1}^{n-1}c_j  (Z_n-Z_{n-j})+c_n^n(Z_n-Z_0) \right)=b(X^{(1)}_n)-
b(X^{(2)}_n).
\]
Applying the first claim in Theorem \ref{thm:implicitcomp} for $E(u)=\frac{1}{2}u^2$, one has almost surely that
\begin{gather}\label{eq:directineqZn}
(\mathcal{D}^{\alpha}|Z|^2)_n
\le - 2 Z_n\cdot (b(X^{(1)}_n)-
b(X^{(2)}_n))\le -2\mu |Z_n|^2.
\end{gather}
The point is that one may pass this inequality somehow to the strong solutions of \eqref{eq:fsde2} by taking $k\to 0$.

\begin{proof}[Proof of Theorem \ref{thm:limitsde}]
Define
\[
u_n :=\mathbb{E}(|Z_n|^2),
\]
and correspondingly
\[
Z(t):=X^{(1)}(t)-X^{(2)}(t),~~u(t):=\mathbb{E}(|Z(t)|^2).
\]

A direct consequence of inequality \eqref{eq:directineqZn} is 
\[
\sup_{n\ge 0}u_n\le u_0.
\]
Applying Theorem \ref{thm:implicitcomp} (2), $u_n \le v_n$, where $v_n$ solves the following induction formula:
\[
(\mathcal{D}^{\alpha}v)_n=-2\mu v_n,~~v_0=u_0=\mathbb{E}(|Z_0|^2).
\]
Let $v(t)$ solve the FODE, $D_c^{\alpha}v=-2\mu v$, ~$v(0)=u_0$. By Theorem \ref{thm:schemeforlinear},
$|v_n-v(t_n)|\le C k^{\alpha}$.

Applying \eqref{eq:momentholder} and Lemma \ref{lmm:convresult}, one has
\[
|u(t_n)-u_n|\le (\sqrt{u}_0+C_1(T))\sqrt{\mathbb{E}(|Z_n-Z(t_n)|^2)} \le C k^{\alpha+H-1}.
\]
Hence, for all $n, nk\le T$, it holds that
\[
u(t_n)\le v(t_n)+C(k^{\alpha+H-1}+k^{\alpha}).
\]
Taking $k\to 0$ then gives
\begin{gather}
u(t)\le v(t)\le u_0 E_{\alpha}(-2\mu t^{\alpha}).
\end{gather}
This inequality clearly implies the claim about the Wasserstein distance using \eqref{eq:W2}.
\end{proof}

As in section \ref{sec:formalfsde}, the overdamped GLE with fractional noise corresponds to 
\begin{gather}
\alpha=2-2H,~\sigma=\frac{\sqrt{2}}{\sqrt{\Gamma(2H+1)}}.
\end{gather}
 We guess that the limiting measure is the Gibbs measure $\pi(x)\propto \exp(-V(x))$.
Rigorously justifying this seems challenging, and we leave it for the future.

\subsection{A numerical simulation}

\begin{figure}
\begin{center}
	\includegraphics[width=0.8\textwidth]{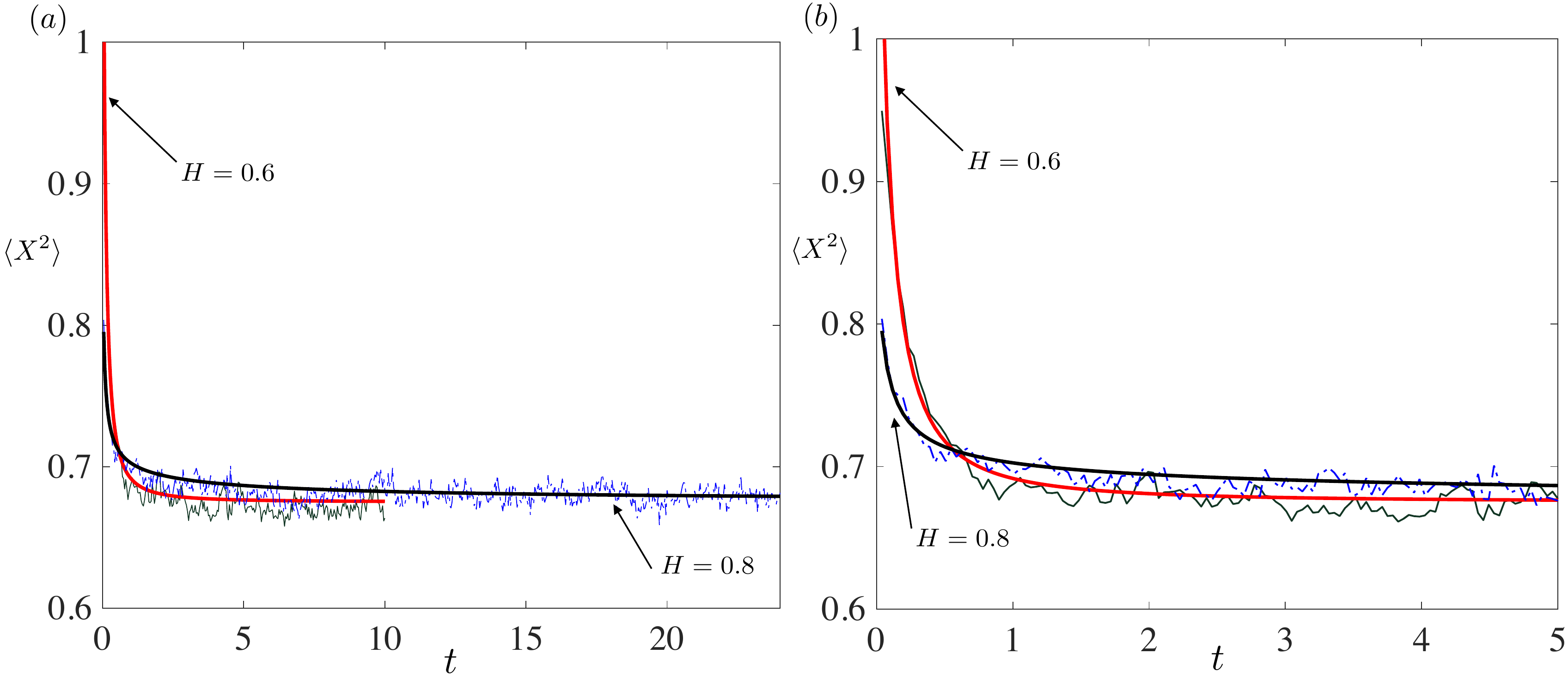}
\end{center}
\caption{Means square distance of the FSDE with potential \eqref{eq:V} for $H=0.6$ and $H=0.8$.}
\label{fig:meansquare}
\end{figure}

In this section, we apply the implicit numerical scheme \eqref{eq:fsdescheme} to a 1D FSDE example with 
\begin{gather}\label{eq:V}
V(x)=\frac{1}{4}x^4.
\end{gather}
We choose $\alpha=2-2H$ so that $G\sim \beta_H B_{1-H}$ as we have mentioned.
This potential is convex but not strongly convex, and the corresponding force $-\nabla V(x)=-x^3$ is nonlinear. Justification of convergence to a limiting measure is by no means easy, not to mention whether the limiting measure is the Gibbs distribution
\[
\pi(x)\propto \exp\left(-\frac{1}{4}x^4\right).
\]

Fig. \ref{fig:meansquare} shows the trend of mean square distance $\mathbb{E}X^2=:\langle X^2\rangle$ with $X(0)=X_0=1$. Fig.  \ref{fig:meansquare} (b) enlarges the portion $t\in [0, 5]$ of Fig.  \ref{fig:meansquare} (a). If the distribution of the FSDE converges to the Gibbs measure, then $\mathbb{E}X^2\to 0.675$.
In the figures, the green curve is the numerical simulation for $H=0.6$ with $k=5/2^7\approx 0.391$ and 
$N_s:=10^4$ samples, while the red solid curve is the fitting curve $0.675+0.02(0.12+t)^{4H-4}$.
The blue curve is the numerical simulation for $H=0.8$ with the same $k$ and number of samples, while the black curve is the fitting curve $0.675+0.015(0.05+t)^{4H-4}+0.04(3+t)^{4H-4}$. We use $4H-4$ power to fit because the variance of $X$ in the linear forcing case has been shown to converge with rate $t^{4H-4}$ in \cite{lll17}. In this sense, the rate in Theorem \ref{thm:limitsde} might not be optimal.

\begin{figure}
\begin{center}
	\includegraphics[width=0.8\textwidth]{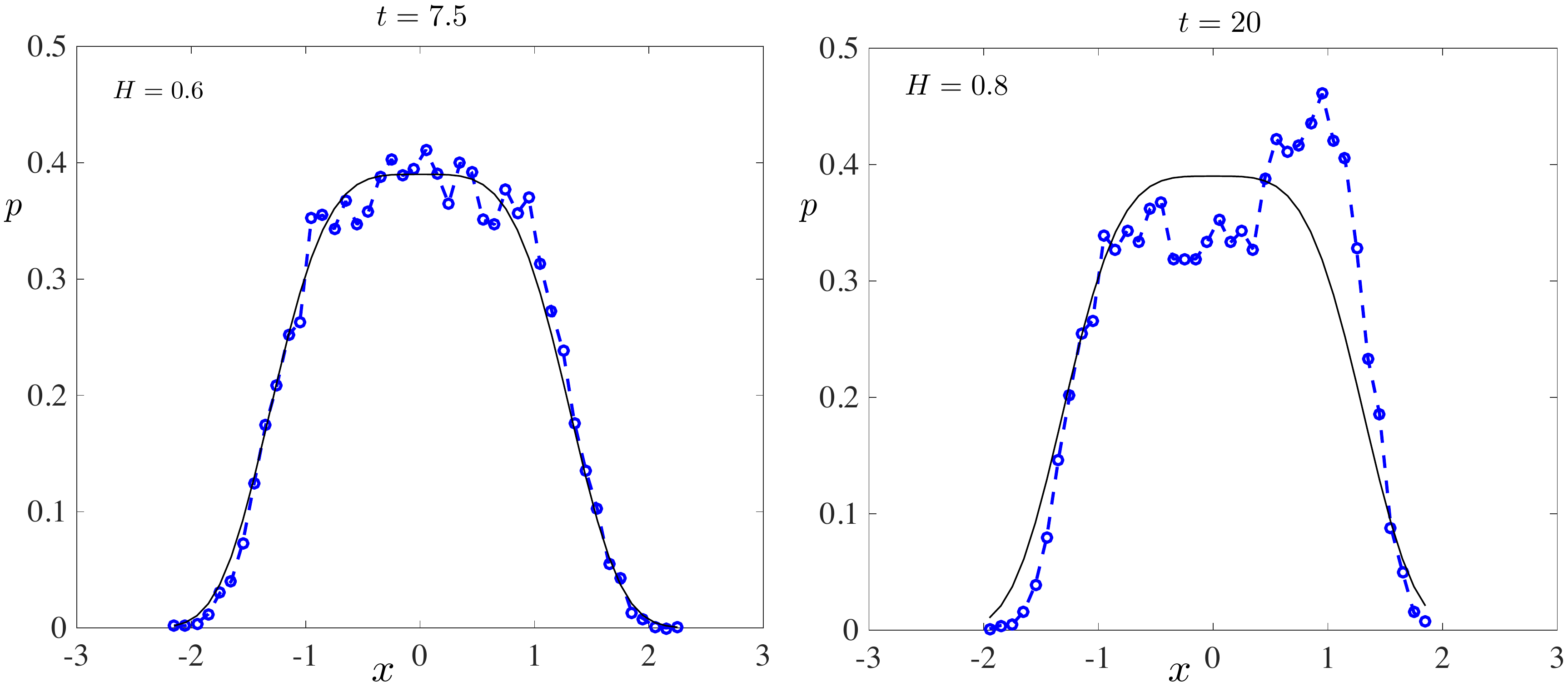}
\end{center}
\caption{Empirical density. When $\alpha=2-2H=0.8$, the convergence is obtained after a reasonable time; when $\alpha=2-2H=0.4$, an intermediate distribution lingers for long time before converging to the final equilibrium.}
\label{fig:density}
\end{figure}

When $H=0.6$ and $\alpha=0.8$, the mean square distance already converges after $t=5$ or so. However, for $H=0.8$ or $\alpha=0.4$, the mean square distance has a rapid drop at the early stage, but then the memory lingers for long time so that the convergence is very slow. This is also the case for normal fractional ODE \cite{fllx18}. In fact, in Fig. \ref{fig:density}, we plot the empirical density versus the Gibbs measure (in black line). When $\alpha=2-2H=0.8$, the distribution is already close to the desired Gibbs distribution at $t=7.5$. However, when $\alpha=2-2H=0.4$, the distribution is roughly like the Gibbs distribution but still has some difference even at $t=20$. In fact, this kind of distribution stays for long time ($t\sim 10^2$). We expect that when $t$ is very large, it can be close to the Gibbs distribution as in \cite{fang2018}. We choose not to do the simulation for $H=0.8$ and $t\sim 10^3$ since the complexity for obtaining a sample path is $O(N^2)$ and the simulation is expensive (we need $10^4$ samples).  For long time simulation when $H$ is close to $1$ (like $H=0.8$ for $T\gtrsim 100$), it is good to adopt the fast scheme in \cite{fang2018}. However, our scheme here is appropriate for dissipative problems due to its good stability properties, and can be used to analyze the time continuous problems compared with the one in \cite{fang2018}.

\section{Time fractional gradient flows}\label{sec:fgradientflow}

In this section, we investigate the time fractional gradient flows using the implicit scheme based on our discretization \eqref{eq:capnewdis2} and establish the error estimates of the numerical scheme. We will use $\langle \cdot, \cdot\rangle$ to denote the inner product in $H$ and $\|\cdot\|$ to denote the norm on $H$. 
We will focus on convex functionals $\phi$:
\begin{assumption}\label{ass:conv}
Suppose the functional $\phi$ is lower semi-continuous, convex and 
$\inf_{u\in \mathbb{R}^d} \phi(u)>-\infty$.
\end{assumption}
\begin{remark}
All the claims in section \ref{sec:fgradientflow} regarding convex functionals have analogies for $\lambda$-convex functionals (i.e. $\exists \lambda\ge 0$, $u\mapsto \phi(u)+\frac{\lambda}{2}|u|^2$ is convex; of course, the proof is more involved). Considering clarity of presentation, we only focus on convex functionals.
\end{remark}

The Frech\'et subdifferential of convex $\phi$ satisfies
\begin{gather}
\xi\in \partial\phi(u) \Leftrightarrow \partial\phi(v)\neq \emptyset, \forall w\in H, \phi(w)-\phi(v)
-\langle \xi, w-v\rangle \ge 0.
\end{gather}
The following strong-weak closure property is a straightforward consequence of this characterization. 
\begin{lemma}
Suppose Assumption \ref{ass:conv} holds. Assume sequences $\{\xi_n\}$ and $\{u_n\}$ satisfy $\xi_n \in \partial\phi(u_n)$ for all $n$, $u_n \to u$ strongly, and that $\xi_n \rightharpoonup \xi$ weakly.
Then $\xi\in \partial\phi(u)$.
\end{lemma}

Fix time $T>0$. Similarly as in \cite[Definition 2.2]{crandall1971}, we define
the following.
\begin{definition}\label{def:strongsol}
$u\in L_{\loc}^1([0, T), H)$ is called a strong solution to \eqref{eq:strongsol}, if  (i) $D_c^{\alpha}u$  is locally integrable on $[0, T)$; (ii)  $\lim_{t\to 0^+}\frac{1}{t}\int_0^t \|u(s)-u_0\|\,ds=0$  (iii) For almost every $t\in [0, T)$, we have $D_c^{\alpha}u\in -\partial\phi(u)$.
\end{definition}

\begin{remark}
If $\alpha=1$, the local integrability of the distributional derivative $Du$ clearly implies that $u$ is absolutely continuous on $[0, T_1]$ for any $T_1\in (0, T)$. The conditions (i)-(iii) in \cite[Definition 2.2]{crandall1971} are automatically satisfied. For $\alpha\in (0, 1)$, imposing $D_c^{\alpha}u\in L_{\loc}^1[0, T)$ does not ensure the uniqueness of $u_0$ (see \cite{liliu2018compact}). To kill this ambiguity, we impose $\lim_{t\to 0^+}\frac{1}{t}\int_0^t \|u(s)-u_0\|\,ds=0$.
\end{remark}

We aim to approximate the solutions of \eqref{eq:strongsol} (though the existence is unclear at this point), following  the method of De Giorgi \cite{de1993new,rossi2006}.
\begin{gather}\label{eq:discreteU}
U_n=\argmin\Big(\frac{1}{2k^{\alpha}}(\sum_{j=1}^{n-1}c_j\|u-U_{n-j}\|^2+c_n^n\|u-U_0\|^2)+\phi(u)\Big).
\end{gather}
Note that the functional on the right hand side of \eqref{eq:discreteU} is the sum of a convex function and some quadratic functionas. Then
\begin{gather}\label{eq:subdiffsum}
\partial \Big(\frac{1}{2k^{\alpha}}(\sum_{j=1}^{n-1}c_j\|u-U_{n-j}\|^2+c_n^n\|u-U_0\|^2)+\phi(u)\Big)
=k^{-\alpha}\Big(c_0u-(\sum_{j=1}^{n-1}c_jU_{n-j}+c_n^nU_0)\Big)+\partial\phi(u),
\end{gather}
and the numerical solution satisfies
\begin{gather}\label{eq:discretescheme}
 -\xi_n :=(\mathcal{D}^{\alpha}U)_n \in  -\partial\phi(U_n).
\end{gather}

Motivated by the proof of Theorem \ref{thm:schemeforlinear}, we consider the set of time steps
\begin{gather}
E_T=\{k>0: k=2^{-m}T,~m\in\mathbb{N} \}.
\end{gather}
The following results from \eqref{eq:subdiffsum} and the strong convexity of the functional in \eqref{eq:discreteU} (proof omitted).
\begin{lemma}\label{lmm:equivfortwominimalsol}
Suppose Assumption \ref{ass:conv} holds. Then, for sufficiently small $k\in E_T$, the discrete schemes \eqref{eq:discreteU} and \eqref{eq:discretescheme} are equivalent and they have a unique solution $\{U^n\}$.
\end{lemma} 

\subsection{Properties of the discrete solutions}

Consider the solution given by \eqref{eq:discretescheme}. Define the function $V(t)$ such that
\begin{gather}
V(t) := -\xi_n,~t\in (t_{n-1}, t_n].
\end{gather}
Using the function $V$, define a natural continuous version interpolation of $U_n$ by 
\begin{gather}\label{eq:continuousextenU}
U(t)=U_0+\frac{1}{\Gamma(\alpha)}\int_0^t (t-s)^{\alpha-1}V(s)\,ds
\end{gather}
with $U(t_m)=U_m$. This continuous interplocation justifies why the discretization \eqref{eq:capnewdis2} is suitable for \eqref{eq:strongsol}.  By \eqref{eq:grouprelation} and the generalized definition (Definition \ref{def:caputo}), one has
\begin{gather}
D_c^{\alpha}U(t)=V(t).
\end{gather}

\begin{lemma}\label{lmm:controlfracint}
Assume Assumption \ref{ass:conv}. For $k \in E_T$ small enough, $\sup_{n: nk\le T}|\phi(U_n)|\le C(U_0, T)$ and
\begin{gather}
\sup_{t\le T}\frac{1}{\Gamma(\alpha)}\int_0^t(t-s)^{\alpha-1}\|V(s)\|^2\,ds\le C(U_0, T).
\end{gather}
\end{lemma}

\begin{proof}
Paring with $\xi_n=-(\mathcal{D}^{\alpha}U)_n\in \partial\phi(U_n)$, and noting $\langle \xi_n, U_n-U_{j}\rangle \ge \phi(U_n)-\phi(U_j)$, one has
\[
-\|\xi_n\|^2=k^{-\alpha}\left(\sum_{j=1}^{n-1}c_j\langle \xi_n, U_n-U_{n-j}\rangle
+c_n^n\langle \xi_n, U_n-U_0\rangle \right)
\ge (\mathcal{D}^{\alpha}\phi(U))_n.
\]
Using the equivalence between \eqref{eq:capnewdis} and \eqref{eq:discretefracint}, and nonnegativity of $\{a_m\}$, one has
\[
\phi(U_n)-\phi(U_0)\le -k^{\alpha}\sum_{m=1}^n a_{n-m}\|\xi_m\|^2
= -\frac{1}{\Gamma(\alpha)}\int_0^{t_n}(t_n-s)^{\alpha-1}\|V(s)\|^2\,ds.
\]
The first claim and the second claim with $t=t_n$ holds.
For general $t\in (t_{n-1}, t_{n})$, the following trivial observation with the result just proved yields the claim in the statement of the lemma.
\begin{gather*}
\begin{split}
\int_0^{t}(t-s)^{\alpha-1}\|V(s)\|^2\,ds
& =\int_0^{t_{n-1}}(t-s)^{\alpha-1}\|V(s)\|^2\,ds
+\|\xi_{n}\|^2\int_{t_{n-1}}^{t}(t-s)^{\alpha-1}\,ds\\
& \le \int_0^{t_{n-1}}(t_{n-1}-s)^{\alpha-1}\|V(s)\|^2\,ds
+\|\xi_{n}\|^2\int_{t_{n-1}}^{t_{n}}(t_{n}-s)^{\alpha-1}\,ds \\
& \le \int_0^{t_{n-1}}(t_{n-1}-s)^{\alpha-1}\|V(s)\|^2\,ds+\int_0^{t_n}(t_n-s)^{\alpha-1}\|V(s)\|^2\,ds.
\end{split}
\end{gather*}

\end{proof}

Now, we compare the numerical solutions with different time steps.
\begin{lemma}\label{lmm:functionseq}
There exists $C(T, U_0)$ independent of $k$ such that when $k$ is small enough, 
\begin{gather}
\|U(t)-U(t+\delta)\|\le C(T, U_0)|\delta|^{\alpha/2},~\text{if}~\max(t, t+\delta)\le T.
\end{gather} 
Let $U_{i}(t)$ be two functions given by \eqref{eq:continuousextenU} for step sizes $k_i\in E_T$ ($i=1,2$).  Then,
\begin{gather}
\sup_{0\le t\le T}\|U_1(t)-U_2(t)\|^2
\le C(U_0,\alpha)(k_1^{\alpha/2}+k_2^{\alpha/2}).
\end{gather}
\end{lemma}

\begin{proof}

Without loss of generality, we assume $\delta>0$. Then, by \eqref{eq:continuousextenU},
\[
\begin{split}
\|U(t)-U(t+\delta)\| \le &~\frac{1}{\Gamma(\alpha)}\Big(\int_0^t [(t-s)^{\alpha-1}-(t+\delta-s)^{\alpha-1}]\|V(s)\|\,ds \\
&+\int_{t}^{t+\delta}(t+\delta-s)^{\alpha-1}\|V(s)\|\,ds\Big)=:I_1+I_2.
\end{split}
\]
The second term is estimated easily by 
\begin{gather*}
I_2\le \frac{1}{\Gamma(\alpha)}\left(\int_t^{t+\delta}\|V(s)\|^2(t+\delta-s)^{\alpha-1}\,ds \right)^{1/2}
\left(\int_t^{t+\delta}(t+\delta-s)^{\alpha-1}\,ds \right)^{1/2}\le C\delta^{\alpha/2}.
\end{gather*}
For the first term $I_1$, we have by H\"older inequality:
\begin{gather}\label{eq:discreteaux1}
I_1\le \frac{1}{\Gamma(\alpha)}\left(\int_0^t(t-s)^{\alpha-1}\Big(1-(\frac{t+\delta-s}{t-s})^{\alpha-1}\Big)^2\,ds\right)^{1/2}
\left(\int_0^t (t-s)^{\alpha-1}\|V(s)\|^2\,ds \right)^{1/2}.
\end{gather}
Clearly,
\begin{gather}\label{eq:discreteaux2}
\int_0^t(t-s)^{\alpha-1}\Big(1-(\frac{t+\delta-s}{t-s})^{\alpha-1}\Big)^2\,ds
\le \int_0^t((t-s)^{\alpha-1}-(t+\delta-s)^{\alpha-1})\,ds 
\le C\delta^{\alpha}.
\end{gather}
The claim follows.

To compare the numerical solutions with steps $k_i,i=1,2$  ($V_i$, $\xi_n^{(i)}$, $i=1,2$ similarly defined), we fix $t\in [0, T]$. Then, there exist $n_1$ and $n_2$ such that
$t\in ((n_1-1)k_1, n_1k_1]\cap ((n_2-2)k_2, n_2k_2]$ and such that $V_i(t)=-\xi^{(i)}_{n_i}$.
Denote
\[
\Delta_i(t) :=U^{(i)}_{n_i}-U_i(t),~~i=1,2.
\]

By the definition of $U_i(t)$ and convexity of $\phi$, one has
\begin{gather}\label{eq:fractionalgradientkeyestimate}
\begin{split}
\langle D_c^{\gamma}U_1-D_c^{\gamma}U_2, U_1-U_2\rangle
 &=\langle V_1(t)-V_2(t), U_1(t)-U_2(t)\rangle \\
 & =-\langle \xi^{(1)}_{n_1}-\xi^{(2)}_{n_2}, U^{(1)}_{n_1}-U^{(2)}_{n_2}\rangle+R(t) \le R(t),
\end{split}
\end{gather}
where $\langle \xi^{(1)}_{n_1}-\xi^{(2)}_{n_2}, U^{(1)}_{n_1}-U^{(2)}_{n_2}\rangle \ge 0$ by convexity of $\phi$ and
\begin{gather*}
R(t)=-\langle V_1(t)-V_2(t), \Delta_1(t)\rangle
+\langle V_1(t)-V_2(t), \Delta_2(t)\rangle .
\end{gather*}
It suffices to estimate $\Delta_i(t)$. We take $i=1$ as the example.
By the definition of $U_1(t)$, 
\begin{gather*}
\Gamma(\alpha)\Delta_1(t)
=-\int_{t}^{n_1k_1}(n_1k_1-s)^{\alpha-1}\xi^{(1)}_{n_1}\,ds
+\int_{0}^{t}[(n_1k_1-s)^{\alpha-1}-(t-s)^{\alpha-1}]V_1(s)\,ds=:I_1^1+I_1^2.
\end{gather*}
The terms corresponding to $I_1^1$ are controlled by (noting $0< n_1k_1-t\le k_1$)
\begin{gather*}
-\langle V_1(t)-V_2(t), I_1^1\rangle=
\langle \xi^{(1)}_{n_1}-\xi^{(2)}_{n_2}, I_1^1\rangle
=\frac{(n_1k_1-t)^{\alpha}}{\Gamma(1+\alpha)}\langle \xi^{(2)}_{n_2}-\xi^{(1)}_{n_1}, \xi^{(1)}_{n_1}\rangle 
\le Ck_1^{\alpha}(\|V_1(t)\|^2+\|V_2(t)\|^2).
\end{gather*}
The terms corresponding to $I_1^2$ can be estimated similarly as in \eqref{eq:discreteaux1}-\eqref{eq:discreteaux2}.
\begin{gather*}
\begin{split}
\langle \xi^{(1)}_{n_1}-\xi^{(2)}_{n_2}, I_1^2\rangle
& \le \frac{1}{\Gamma(\alpha)}\|V_1(t)-V_2(t)\|
\int_0^t[(t-s)^{\alpha-1}-(n_1k_1-s)^{\alpha-1}]\|V_1(s)\|\,ds\\
& \le C(\alpha, U_0)\|V_1(t)-V_2(t)\| k_1^{\alpha/2}.
\end{split}
\end{gather*}

By the explicit formula \eqref{eq:continuousextenU},   $U_i(t)$ is absolutely continuous. Proposition 3.11 in \cite{liliu2018} can be easily generalized to show that
\[
\frac{1}{2}D_c^{\alpha}\|U_1(t)-U_2(t)\|^2
\le \langle D_c^{\alpha}(U_1-U_2)(t), U_1(t)-U_2(t)\rangle.
\]

Overall,
\[
D_c^{\alpha}(\|U_1(t)-U_2(t)\|^2)
\le C(k_1^{\alpha/2}+k_2^{\alpha/2})(\|V_1(t)\|^2
+\|V_2(t)\|^2).
\]
Lemma \ref{lmm:controlfracint} then yields the result.
\end{proof}

\subsection{Well-posedness and numerical error estimates}

In this subsection, we establish the existence and uniqueness of the time fractional gradient flow under some assumptions and give the error estimate of the numerical scheme. Besides Assumption \ref{ass:conv}, we also need a certain regularity property of the subdifferential mapping $\partial\phi$. 
\begin{assumption}\label{ass:lambdaconv}
Let Assumption \ref{ass:conv} hold. Moreover, $v_n\to v$ strongly implies any sequence $\{\xi_n\}$ with $\xi_n\in \partial\phi(v_n)$ converges weakly to some $\xi\in H$.
\end{assumption}

\begin{theorem}\label{thm:fracgradient}
Suppose Assumption \ref{ass:lambdaconv} holds. For any $T>0$, the fractional gradient flow \eqref{eq:strongsol} has a unique strong solution $u$ on $[0, T)$ in the sense of Definition \eqref{def:strongsol}. The strong solution is H\"older continuous on $[0, T)$:
\begin{gather}
\|u(t+\delta)-u(t)\|\le C|\delta|^{\alpha/2}.
\end{gather}
Besides, we have the following error estimates for the numerical solution \eqref{eq:discreteU}:
\begin{gather}
\sup_{n: nk\le T}\|U_n-u(t_n)\|\le C(\alpha, T)k^{\alpha/4}.
\end{gather}
\end{theorem}
\begin{proof}[Proof of Theorem \ref{thm:fracgradient}]
By Lemma \ref{lmm:functionseq}, the family $\{ U_i(\cdot): i\in E_T\}$ is a Cauchy sequence in $C([0, T], H)$.
 Therefore, there exists $u(\cdot)\in C([0, T]; H)$, such that $U_i(\cdot)$ converges to $u$ in $C([0, T]; H)$.

Consider the piecewise constant interpolation of the numerical fractional derivative, $\{V_i(t)\}$.
Lemma \ref{lmm:controlfracint} implies that  $V_i(t)\in L^2(0, T)$ with the $L^2(0, T)$ norm uniformly bounded. Then, there is a further subsequence so that we have the weak convergent sequence $V_i(t) \rightharpoonup v(t)$ in $L^2(0, T; H)$, with the estimate
\[
\sup_{t\le T}\frac{1}{\Gamma(\alpha)}\int_0^t(t-s)^{\alpha-1}\|v(s)\|^2\,ds\le C_2(U_0, T).
\]
With the convergence in hand, $u(t)=u_0+\frac{1}{\Gamma(\alpha)}\int_0^t(t-s)^{\alpha-1}v(s)\,ds$.
Hence, in the distributional sense,
\begin{gather}\label{eq:caputoutov}
D_c^{\alpha}u=v(t).
\end{gather}
Similarly as in the proof of Lemma \ref{lmm:functionseq}, $u(\cdot)$ is H\"older continuous
with order at least $\alpha/2$.

By Lemma \ref{lmm:controlfracint}, we have found that $\phi(U_n)$ is uniformly bounded. 
Now we consider piecewise linear interpolation of $\{U_n\}$, denoted by $\bar{U}_i(\cdot)$,
\[
\bar{U}_i(t)=U_n^{(i)},~~t\in (t_{n-1}, t_n].
\]
Then, $V_i(t)\in -\partial\phi(\bar{U}_i(t)),~\forall t\in [0, T),~\forall i$.  By the uniform H\"older continuity of $U_i(t)$ in Lemma \ref{lmm:functionseq}  and $\bar{U}_i(t)=U_i(t_n)$
for $t\in (t_{n-1}, t_n]$, we find that for all $t$, $\bar{U}_i(t)$ converges
strongly to $u(t)$.  With Assumption \ref{ass:lambdaconv}, we find that for all $t$,
\begin{gather}\label{eq:pointwiseweak}
V_i(t) \rightharpoonup \bar{v}(t),~H,
\end{gather}
so that
\begin{gather}\label{eq:limitsubgrad}
\bar{v}(t)\in -\partial\phi(u(t)),
\end{gather}
by the weak--strong closure property.

Hence, for any $w\in H$ so that $w1_{[0, T]} \in L^2(0, T; H)$, we find that
$\langle V_i(t), w\rangle$ converges in $L^2(0, T; \mathbb{R})$ to $\langle v(t), w\rangle$.
By \eqref{eq:pointwiseweak}, we also have $\langle V_i(t), w\rangle \to \langle \bar{v}(t), w\rangle$ 
for all $t\in [0, T]$. Consequently, for a.e. $t\in [0, T]$, we have
\[
\langle v(t), w\rangle = \langle \bar{v}(t), w\rangle.
\]
Since $H$ is separable, we can find a basis consisting of countable elements $\{w_n\}$. Consequently, we have
\[
\langle v(t)-\bar{v}(t), w_n\rangle =0,~\forall n, a.e. t\in [0, T]
\]
Moreover, $v-\bar{v}\in L^2(0, T; H')$ (where $H'$ means $H$ equipped with the weak star topology), we have then
\[
\int_0^T \langle v(t)-\bar{v}(t), w(t)\rangle\,dt=0
\]
for $w(t)$ to be simple functions, and then $L^2(0, T; H)$ functions. Hence, $v(t)=\bar{v}(t)$ for a.e. $t\in [0, T]$.
By \eqref{eq:limitsubgrad} and \eqref{eq:caputoutov}, we find that $u(t)$ is a strong solution under Definition \ref{def:strongsol}. Taking limit $k_2\to 0$ in $\|U_1(t)-U_2(t)\|^2\le C(U_0, \alpha)(k_1^{\alpha/2}+k_2^{\alpha/2})$ yields $\|U_n(t) -u(t)\|^2 \le C(U_0, \alpha)k^{\alpha/2}$.

For the uniqueness, suppose we have two strong solutions $u_i(t),i=1,2$ such that $-\xi_i(t):=D_c^{\alpha}u_i(t)\in -\partial\phi(u_i(t))$. Then, $\langle u_1-u_2, D_c^{\alpha}(u_1-u_2)\rangle\le 0$. Some regularization procedure can yield $\frac{1}{2}D_c^{\alpha}\|u_1-u_2\|^2\le \langle u_1-u_2, D_c^{\alpha}(u_1-u_2)\rangle\le 0$. This then implies uniqueness.
\end{proof}

As a concluding remark, the orders of estimates in Theorem \ref{thm:fracgradient} are not optimal. If one can show that $\xi_n$ is bounded, then one can improve the orders.  Lastly, we give a quick glimpse of the case $H=\mathbb{R}^d$, $\phi\in C^1(\mathbb{R}^d)$ (instead of requiring $\nabla\phi$ to be Lipschitz as in \cite{liliu2018,fllx18}) so that \eqref{eq:strongsol} becomes the FODE:
\begin{gather}
D_c^{\alpha}u=-\nabla \phi(u).
\end{gather}
The following asymptotic behavior holds when $\phi$ is strongly convex.
\begin{proposition}\label{pro:fracgradientsmooth}
Assume that $\phi\in C^1(\mathbb{R}^d)$ and $\phi-\frac{\mu}{2}|u|^2$ is convex for some $\mu>0$. 
Let $u^*$ be the global minimizer of $\phi$. Then, for some $C$ depending on $u_0,\mu$,
\[
\phi(u(t))-\phi(u^*)\le (\phi(u_0)-\phi(u^*))E_{\alpha}\left(-Ct^{\alpha} \right).
\]
Moreover, $|u(t)-u^*|\le C(1+t)^{-\alpha/2}$.
\end{proposition}
\begin{proof}

Consider the implicit scheme
\[
(\mathcal{D}^{\alpha}U)_n=-\nabla\phi(U_n).
\]
Using the first claim in Theorem \ref{thm:implicitcomp}, one has
\[
(\mathcal{D}^{\alpha}\phi(U))_n\le -|\nabla\phi(U_n)|^2\le 0.
\]
Hence $\phi(U_n)$ is bounded. Since $\phi$ is strongly convex and thus $\lim_{R\to\infty}\inf_{|u|\ge R}\phi(u)=+\infty$,  $\{U_n\}$ is in a compact domain $K$ that only depends on $u_0$.
By Theorem \ref{thm:implicitcomp} and the Polyak-Lojasiewicz inequality ($|\nabla \phi(x)|^2\ge 2\mu(\phi(x)-\phi(u^*))$), one has
\[
(\mathcal{D}^{\alpha}(\phi(U)-\phi(u^*)))_n=(\mathcal{D}^{\alpha}\phi(U))_n\le -|\nabla\phi(U_n)|^2\le -2\mu(\phi(U_n)-\phi(u^*)).
\]
The Polyak-Lojasiewicz inequality is obtained by 
\[
\phi(y)\ge \phi(u)+\nabla\phi(u)\cdot(y-u)+\frac{\mu}{2}|y-u|^2\ge \phi(u)-\frac{1}{2\mu}|\nabla \phi(u)|^2.
\]

By the second claim in Theorem \ref{thm:implicitcomp} and the third claim in Theorem \ref{thm:schemeforlinear}, it holds for any $nk\le T$ that
\[
\phi(U_n)-\phi(u^*)\le (\phi(u_0)-\phi(u^*))E_{\alpha}(-2\mu(nk)^{\alpha})+o(k).
\]
Taking $k\to 0$ and by Theorem \ref{thm:fracgradient} (convergence and continuity of $u(t)$), one thus has for any $t\le T$:
\[
\phi(u(t))-\phi(u^*)\le (\phi(u_0)-\phi(u^*))E_{\alpha}(-2\mu t^{\alpha})
\]
Since $T$ is arbitrary, the first claim is true for all $t$. 

Similarly,  $(\mathcal{D}^{\alpha}(U-u^*)^2)_n
\le 2\langle U_n-u^*, -\nabla\phi(U_n)\rangle \le -2\mu |U_n-u^*|^2$.
Theorem \ref{thm:fracgradient} allows us to take $k\to 0$ to obtain $|u(t)-u^*|\le |u_0-u^*|\sqrt{E_{\alpha}(-2\mu t^{\alpha})}$.
Since $E_{\alpha}(-s)\sim C_1s^{-1}$ as $s\to\infty$, the second claim follows.
\end{proof}

\section*{Acknowledgement}
The work of L. Li was partially sponsored by Shanghai Sailing Program 19YF1421300. The work of J.-G. Liu was partially supported by KI-Net NSF RNMS11-07444 and NSF DMS-1812573.

\bibliographystyle{unsrt}
\bibliography{fracgradient}

\end{document}